\documentclass[11pt]{article}
\usepackage[utf8x]{inputenc}
\usepackage{quoting}
\newcommand{\vertiii}[1]{{\left\vert\kern-0.25ex\left\vert\kern-0.25ex\left\vert #1
    \right\vert\kern-0.25ex\right\vert\kern-0.25ex\right\vert}}
\quotingsetup{vskip=0pt,leftmargin=10pt,rightmargin=10pt}

\def\showauthornotes{1}
\def\showkeys{0}
\def\showdraftbox{0}
\def\usemicrotype{1}
\def\showfixme{0}

\usepackage{etex}

\usepackage[l2tabu, orthodox]{nag}

\usepackage{xspace,enumerate}

\usepackage[dvipsnames]{xcolor}

\usepackage[T1]{fontenc}
\usepackage[full]{textcomp}

\usepackage[american]{babel}

\usepackage{mathtools}

\usepackage{amsthm}

\newtheorem{theorem}{Theorem}[section]
\newtheorem*{theorem*}{Theorem}

\newtheorem*{proposition*}{Proposition}
\newtheorem{lemma}[theorem]{Lemma}
\newtheorem*{lemma*}{Lemma}

\newtheorem*{conjecture*}{Conjecture}

\newtheorem*{fact*}{Fact}

\newtheorem*{hypothesis*}{Hypothesis}

\theoremstyle{definition}

\theoremstyle{remark}

\newtheorem*{claim*}{Claim}
\newtheorem{remark}[theorem]{Remark}
\newtheorem*{remark*}{Remark}

\newtheorem*{observation*}{Observation}

\usepackage[
letterpaper,
top=0.7in,
bottom=0.9in,
left=1in,
right=1in]{geometry}

\usepackage[varg]{pxfonts}
\usepackage{textcomp}
\usepackage[scr=rsfso]{mathalfa}
\usepackage{bm} 
\let\mathbb\varmathbb

\ifnum\showkeys=1
\usepackage[color]{showkeys}
\fi

\usepackage{hyperref}
\hypersetup{
pagebackref,
colorlinks=true,
urlcolor=blue,
linkcolor=blue,
citecolor=OliveGreen,
}
\usepackage{prettyref}

\newcommand{\savehyperref}[2]{\texorpdfstring{\hyperref[#1]{#2}}{#2}}

\newrefformat{eq}{\savehyperref{#1}{\textup{(\ref*{#1})}}}
\newrefformat{lem}{\savehyperref{#1}{Lemma~\ref*{#1}}}
\newrefformat{def}{\savehyperref{#1}{Definition~\ref*{#1}}}
\newrefformat{thm}{\savehyperref{#1}{Theorem~\ref*{#1}}}
\newrefformat{cor}{\savehyperref{#1}{Corollary~\ref*{#1}}}
\newrefformat{cha}{\savehyperref{#1}{Chapter~\ref*{#1}}}
\newrefformat{sec}{\savehyperref{#1}{Section~\ref*{#1}}}
\newrefformat{app}{\savehyperref{#1}{Appendix~\ref*{#1}}}
\newrefformat{tab}{\savehyperref{#1}{Table~\ref*{#1}}}
\newrefformat{fig}{\savehyperref{#1}{Figure~\ref*{#1}}}
\newrefformat{hyp}{\savehyperref{#1}{Hypothesis~\ref*{#1}}}
\newrefformat{alg}{\savehyperref{#1}{Algorithm~\ref*{#1}}}
\newrefformat{rem}{\savehyperref{#1}{Remark~\ref*{#1}}}
\newrefformat{item}{\savehyperref{#1}{Item~\ref*{#1}}}
\newrefformat{step}{\savehyperref{#1}{step~\ref*{#1}}}
\newrefformat{conj}{\savehyperref{#1}{Conjecture~\ref*{#1}}}
\newrefformat{fact}{\savehyperref{#1}{Fact~\ref*{#1}}}
\newrefformat{prop}{\savehyperref{#1}{Proposition~\ref*{#1}}}
\newrefformat{prob}{\savehyperref{#1}{Problem~\ref*{#1}}}
\newrefformat{claim}{\savehyperref{#1}{Claim~\ref*{#1}}}
\newrefformat{relax}{\savehyperref{#1}{Relaxation~\ref*{#1}}}
\newrefformat{red}{\savehyperref{#1}{Reduction~\ref*{#1}}}
\newrefformat{part}{\savehyperref{#1}{Part~\ref*{#1}}}

\newcommand{\Sref}[1]{\hyperref[#1]{\S\ref*{#1}}}

\usepackage{nicefrac}

\ifnum\usemicrotype=1
\usepackage{microtype}
\fi

\ifnum\showauthornotes=1
\newcommand{\Authornote}[2]{{\sffamily\small\color{red}{[#1: #2]}}}
\newcommand{\Authornotecolored}[3]{{\sffamily\small\color{#1}{[#2: #3]}}}
\newcommand{\Authorcomment}[2]{{\sffamily\small\color{gray}{[#1: #2]}}}
\newcommand{\Authorstartcomment}[1]{\sffamily\small\color{gray}[#1: }

\newcommand{\Authorfnote}[2]{\footnote{\color{red}{#1: #2}}}
\newcommand{\Authorfixme}[1]{\Authornote{#1}{\textbf{??}}}
\newcommand{\Authormarginmark}[1]{\marginpar{\textcolor{red}{\fbox{\Large #1:!}}}}
\else
\newcommand{\Authornote}[2]{}
\newcommand{\Authornotecolored}[3]{}
\newcommand{\Authorcomment}[2]{}
\newcommand{\Authorstartcomment}[1]{}

\newcommand{\Authorfnote}[2]{}
\newcommand{\Authorfixme}[1]{}
\newcommand{\Authormarginmark}[1]{}
\fi

\ifnum\showfixme=0

\fi

\usepackage{boxedminipage}

\newcommand{\Esymb}{\mathbb{E}}

\DeclareMathOperator*{\E}{\Esymb}

\newcommand\bdot\bullet

\ifx\mathds\undefined 

\else

\fi

\newcommand{\C}{\mathbb C}

\renewcommand{\leq}{\leqslant}

\renewcommand{\geq}{\geqslant}

\ifnum\showdraftbox=1
\newcommand{\draftbox}{\begin{center}
  \fbox{%
    \begin{minipage}{2in}%
      \begin{center}%
          \Large\textsc{Working Draft}\\%
        Please do not distribute%
      \end{center}%
    \end{minipage}%
  }%
\end{center}
\vspace{0.2cm}}
\else
\newcommand{\draftbox}{}
\fi

\let\epsilon=\varepsilon

\numberwithin{equation}{section}

\newcommand\MYcurrentlabel{xxx}

\newcommand{\MYstore}[2]{%
  \global\expandafter \def \csname MYMEMORY #1 \endcsname{#2}%
}

\newcommand{\MYload}[1]{%
  \csname MYMEMORY #1 \endcsname%
}

\newcommand{\MYnewlabel}[1]{%
  \renewcommand\MYcurrentlabel{#1}%
  \MYoldlabel{#1}%
}

\newcommand{\MYdummylabel}[1]{}

\newcommand{\torestate}[1]{%
  \let\MYoldlabel\label%
  \let\label\MYnewlabel%
  #1%
  \MYstore{\MYcurrentlabel}{#1}%
  \let\label\MYoldlabel%
}

\newcommand{\restatetheorem}[1]{%
  \let\MYoldlabel\label
  \let\label\MYdummylabel
  \begin{theorem*}[Restatement of \prettyref{#1}]
    \MYload{#1}
  \end{theorem*}
  \let\label\MYoldlabel
}

\newcommand{\restatelemma}[1]{%
  \let\MYoldlabel\label
  \let\label\MYdummylabel
  \begin{lemma*}[Restatement of \prettyref{#1}]
    \MYload{#1}
  \end{lemma*}
  \let\label\MYoldlabel
}

\newcommand{\restateprop}[1]{%
  \let\MYoldlabel\label
  \let\label\MYdummylabel
  \begin{proposition*}[Restatement of \prettyref{#1}]
    \MYload{#1}
  \end{proposition*}
  \let\label\MYoldlabel
}

\newcommand{\restatefact}[1]{%
  \let\MYoldlabel\label
  \let\label\MYdummylabel
  \begin{fact*}[Restatement of \prettyref{#1}]
    \MYload{#1}
  \end{fact*}
  \let\label\MYoldlabel
}

\newcommand{\restate}[1]{%
  \let\MYoldlabel\label
  \let\label\MYdummylabel
  \MYload{#1}
  \let\label\MYoldlabel
}

\newcommand{\addreferencesection}{
  \phantomsection
  \addcontentsline{toc}{section}{References}
}

\let\origparagraph\paragraph
\renewcommand{\paragraph}[1]{\origparagraph{#1.}}

\allowdisplaybreaks

\usepackage{paralist}

\usepackage{comment}

\let\citet\cite

\theoremstyle{definition}

\usepackage{relsize}
\usepackage[font=footnotesize]{caption}
\usepackage{yfonts}

\DeclareUrlCommand\email{}

\usepackage{appendix}
\usepackage{amssymb}
\usepackage{amsfonts}
\usepackage{latexsym}
\usepackage{amsmath}

\usepackage[T1]{fontenc}
\usepackage{fullpage,appendix}
\usepackage{dsfont}
\usepackage[capitalize]{cleveref}

\usepackage{color}
\usepackage{tikz}

\newcommand{\ignore}[1]{}
\definecolor{corlinks}{RGB}{64,128,128}
\definecolor{cormenu}{RGB}{0,37,94}
\definecolor{corurl}{RGB}{0,46,91}

\renewcommand{\S}{\mathbb{S}}

\renewcommand{\int}{\mathsf{int}}

\makeatletter
\let\orgdescriptionlabel\descriptionlabel
\renewcommand*{\descriptionlabel}[1]{%
  \let\orglabel\label
  \let\label\@gobble
  \phantomsection
  \edef\@currentlabel{#1}%
  \let\label\orglabel
  \orgdescriptionlabel{#1}%
}
\makeatother

\title{On Concentration Inequalities for Random Matrix Products\thanks{T.K. is supported by NSF Grant CCF-1718695. S.M. and N.S. are supported by NSF Grant CCF-1553751.}}

\author{
Tarun Kathuria \\ UC Berkeley \\
\text{tarunkathuria@berkeley.edu} \and Satyaki Mukerjee \\ UC Berkeley \\ \text{satyaki@berkeley.edu} \and Nikhil Srivastava \\ UC Berkeley \\ \text{nikhil@math.berkeley.edu}
}

\begin{document}

\maketitle
\draftbox
\thispagestyle{empty}

\begin{abstract}
    Consider $n$ complex random matrices $X_1,\ldots,X_n$ of size $d\times d$ sampled i.i.d. from a distribution with mean $\E[X]=\mu$. While the concentration of averages of these matrices is well-studied, the concentration of other functions of such matrices is less clear. One function which arises in the context of stochastic iterative algorithms, like Oja's algorithm for Principal Component Analysis, is the normalized matrix product defined as
    \begin{align*}
        \prod\limits_{i=1}^{n}\left(I + \frac{X_i}{n}\right).
    \end{align*}
    Concentration properties of this normlized matrix product were recently studied by \cite{HW19}. However, their result is suboptimal in terms of the dependence on the dimension of the matrices as well as the number of samples. In this paper, we present a stronger concentration result for such matrix products which is optimal in $n$ and $d$ up to constant factors. Our proof is based on considering a matrix Doob martingale, controlling the quadratic variation of that martingale, and applying the Matrix Freedman inequality of Tropp \cite{TroppIntro15}. 
\end{abstract}

\setcounter{page}{1}
\section{Setup}
Suppose ${X_1},\ldots,{X_n} \in \C^{d\times d}$ are random matrices sampled i.i.d from some distribution with $\E[{X_i}]=\mu$ and $\|{X}_i\|_\mathsf{op}\leq L$ almost surely.
A famous result is the matrix Bernstein inequality \cite{TroppIntro15} for sums of random matrices, which in this setting asserts that 
\begin{align*}
    \mathsf{Pr}\left[\left\|\sum\limits_{i=1}^{n}\frac{X_i}{n}-\mu\right\|_\mathsf{op}\geq t\right] \leq 2d \cdot \exp(-n t^2/2L^2),
\end{align*}
whenever $t \leq L\sqrt{\frac{\log d}{n}}$ and $n\geq \log(d)$. For some numerical linear algebra problems, it is of interest to consider instead of sums, functions of the form $$f({X_1},\ldots, {X_n}) = \prod\limits_{i=1}^{n}\left({I} +\frac{{X}_i}{n}\right).$$ We will refer to such functions as matrix product functions. One can easily prove the following lemma
\begin{lemma}$\E_{{X_1},\ldots,{X_n}}[f(X_1,\ldots,X_n)] \preceq e^\mu$ with equality in the limit as $n\rightarrow \infty$.
\end{lemma}
\begin{proof}
\begin{align*}
    \E_{{X_1},\ldots,{X_n}}[f({X_1},\ldots,{X_n})] &=  \E_{{X_1},\ldots,{X_n}}\left[\prod\limits_{i=1}^{n}\left(\mathbf{I} +\frac{{X}_i}{n}\right)\right]\\
    &=\prod\limits_{i=1}^{n}\E_{{X_i}}\left[{I}+\frac{{X_i}}{n}\right]\\
    &=\prod\limits_{i=1}^{n}\left[{I}+\frac{\mu}{n}\right]\\
    &=\left({I}+\frac{\mu}{n}\right)^n
    \preceq e^\mu,
\end{align*}
and there is equality in the limit. The second equality is because of independence of ${X_i}.$
\end{proof}
Recently a central limit theorem for matrix products was established \cite{EH18} and the following concentration inequality was proven by Henriksen and Ward \cite{HW19}. 
\begin{theorem}[\cite{HW19}]Assuming $\max\{3,Le^2\}\leq \log(n)+1\leq \big(\frac{16n}{\log(dne/\delta)}\big)^{1/3}$, we have that with probability greater than $1-2\delta$, the following holds
\begin{align*}
    \|f({X_1,\ldots,X_n)}-e^\mu\|\leq \frac{O(Le^L)\log(n)}{\sqrt{n}}\big(\sqrt{\log(d/\delta)+\log(n)^2}+\frac{\log(n)}{\sqrt{n}}\big)+\frac{L^2e^L}{n}.
\end{align*}
\end{theorem}

 Their proof groups the product into sums of $k-$wise products in a careful way, appealing to Baranyai's theorem, and applies matrix Bernstein inequality to each partition. This approach loses a $(\log n)^2$ factor compared to the matrix Bernstein result for sums and it is unclear whether this is necessary. In this note, we will give a simple proof relying on the Matrix Freedman inequality \cite{TroppIntro15} which does not lose the $\log n$ factors, essentially matching the matrix Bernstein inequality for sums of matrices upto constants. 
\begin{theorem}\label{thm:mainthm}\begin{align*}\mathsf{Pr}\left[\left\|f({X}_1,\ldots,{X}_n)-e^\mu\right\|_\mathsf{op}\geq t\right] \leq 2d \cdot \exp(-cn t^2/L^2e^{2L}),\end{align*}
whenever $t\leq Le^L\sqrt{\frac{\log d}{n}}$, for some absolute constant $c$. Equivalently, for every $\delta\in(0,1)$ with probabiity greater than $1-\delta$, we have
\begin{align*}
    \|f({X_1,\ldots,X_n)}-e^\mu\|\leq \frac{O(Le^L)}{\sqrt{n}}\sqrt{\log(d/\delta)}.
\end{align*}
\end{theorem}
The key difference in this result and the matrix Bernstein inequality for sums is the $L^2e^{2L}$ factor instead of $L^2$. We will later show that even for the special case of products of scalars, such an $e^{O(L)}$ dependence is necessary if the bound is written only in terms of $L$ and not $\mu$. 
\begin{remark}[Independent Work]
The recently posted independent work \cite{newprod} gives a different proof of a more refined version of Theorem \ref{thm:mainthm}, which has slightly better constants and an $L^2e^{2\mu}$ term in the denominator rather than $L^2e^{2L}$ (see their Theorem I). Their approach is also martingale-based, but instead of Matrix Freedman it relies on certain smoothness properties of Schatten norms, also yielding more general results for Schatten norms of matrix products which our proof does not yield.
\end{remark}
\section{Matrix Concentration via Doob Martingale}
Our concentration proof proceeds by constructing a Doob martingale and controlling the norm of each increment and the total predictable variation of the martingale process. Let $$Y_k = \E[f(X_1, ..., X_n) | X_1, ..., X_k] - \E[f(X_1, ..., X_n) | X_1, ..., X_{k-1}],$$ where $f(X_1, ..., X_n) =\prod\limits_{i=1}^{n}\big(I +\frac{X_i}{n}\big)$. Note that $\E[Y_i |  X_1, ..., X_i] = 0$, thus $Y_i$ is a martingale. We also observe that as $X_1, ..., X_n$ are independent, \begin{align*}Y_k &= \E\left[f(X_1, ..., X_n) | X_1, ..., X_k\right] - \E\left[f(X_1, ..., X_n) | X_1, ..., X_{k-1}\right] \\
&=  \prod_{i=1}^k \big(I+\frac{X_i}{n}\big) \prod\limits_{i=k+1}^{n}\E\bigg[\big(I +\frac{X_i}{n}\big)\big] - \prod_{i=1}^{k-1} \big(I + \frac{X_i}{n}\big) \prod\limits_{i=k+1}^{n}\E\bigg[\big(I +\frac{X_i}{n}\big)\bigg] \\ 
&= \prod_{i=1}^{k-1} \big(I+\frac{X_i}{n}\big) \frac{X_k - \mu
}{n}\prod\limits_{i=k+1}^{n}\big(I +\frac{\mu}{n}\big).
\end{align*}  
We thus use submultiplicativity of the spectral norm to obtain,
\begin{align*}
    \|Y_k\|&=\bigg\|\prod_{i=1}^{k-1} \big(I+\frac{X_i}{n}\big) \cdot \frac{X_k - \mu
}{n}\cdot\prod\limits_{i=k+1}^{n}\big(I +\frac{\mu}{n}\big)\bigg\|\\
&\leq \bigg(\prod\limits_{i=1}^{k-1}\bigg\|I+\frac{X_i}{n}\bigg\|\bigg)\bigg\|\frac{X_k-\mu}{n}\bigg\|\bigg(\prod\limits_{i=k+1}^{n}\bigg\|\bigg(I+\frac{\mu}{n}\bigg)\bigg\|\bigg)\\
&\leq \frac{2L}{n}\big(1+\frac{L}{n}\big)^{n-1}\\
&\leq \frac{2Le^L}{n},
\end{align*}
where the second inequality follows from the norms of $X_i$ (and hence norm of $\mu$) being bounded by $L$ almost surely and the last inequality follows as $(1+x/n)^{(n-1)}\leq (1+x/n)^n\leq e^x$ for non-negative $x$.

Also note that \begin{align*}\left\|\mathbb{E}\left[Y_k Y_k^*| X_1,\ldots,X_{k-1}\right]\right\| &= \left\|\left(\prod_{i=1}^{k-1} I+\frac{X_i}{n}\right) \frac{X_k - \mu
}{n}\prod\limits_{i=k+1}^{n}\left(I +\frac{\mu}{n}\right)\prod\limits_{i=n}^{k+1}\left(I +\frac{\mu}{n}\right)\frac{X_k^* - \mu
}{n}\left(\prod_{i=k-1}^{1} I+\frac{X_i^*}{n}\right)\right\| \\
&≤ \prod_{i=1}^{k-1} \left\|I+\frac{X_i}{n}\right\| \cdot \left\|\frac{X_k - \mu
}{n}\right\| \prod\limits_{i=k+1}^{n}\left\|I +\frac{\mu}{n}\right\| \prod\limits_{i=n}^{k+1}\left\|I +\frac{\mu}{n}\right\| \cdot \left\|\frac{X_k^* - \mu
}{n}\right\|\prod_{i=k-1}^{1} \left\|I+\frac{X_i^*}{n} \right\|  \\
&≤ \frac{4L^2}{n^2} \left(1+\frac{L}{n}\right)^{2n-2} \\
&≤ \frac{4L^2}{n^2} e^{2L}. \\
\end {align*} 
Hence, we get that for any $k\leq n$,
\begin{align*}
    \left\|\sum\limits_{i=1}^{k}\mathbb{E}\left[Y_k Y_k^*| X_1,\ldots,X_{k-1}\right]\right\|
    &\leq \sum\limits_{i=1}^{k}\left\|\mathbb{E}\left[Y_k Y_k^*| X_1,\ldots,X_{k-1}\right]\right\|\\
    &\leq \frac{4L^2e^{2L}k}{n^2}\\
    &\leq\frac{4L^2e^{2L}}{n}.
\end{align*}
To conclude the proof, we use the Matrix Freedman inequality \cite{TroppIntro15} for concentration of matrix valued martingales which is stated next.
\begin{theorem}\label{thm:freedman}
Suppose $Y_k=\sum\limits_{i=1}^k X_i $ is a martingale with $d\times d$ matrix increments $X_i$ satisfying $\|X_i\|\leq R$ almost surely. Let the predictable variations of the process be $W_k^{(1)}=\sum\limits_{i=1}^{k}\E[X_i X_i^*|X_1,\ldots,X_{i-1}]$ and $W_k^{(2)}=\sum\limits_{i=1}^{k}\E[X_i^* X_i|X_1,\ldots,X_{i-1}]$. Then for all $t\geq 0$, we have
\begin{align*}
    \mathsf{Pr}[\exists k \geq 0 : \|Y_k\|\geq t \ \mathit{and} \   \max\{\|W_k^{(1)}\|,\|W_k^{(2)}\|\}\leq \sigma^2]\leq 2d\exp\left(-\frac{c t^2}{Rt + \sigma^2}\right).
\end{align*}
\end{theorem}

\begin{proof}[Proof of Theorem \ref{thm:mainthm}]
From the above argument, we get that the increments of our martingale $Y_k$ are bounded by $Le^L/n$ in spectral norm almost surely and that the norm of the predictable quadratic variation (the analysis of $\mathbb{E}[Y_k^* Y_k| X_1,\ldots,X_{k-1}]$ is identical) is bounded by $\frac{4L^2e^{2L}}{n}$ almost surely. Hence we can use Thereom \ref{thm:freedman}, to conclude that

\begin{align*}
     \mathsf{Pr}\left[\|Y_n\|\geq t \right]&\leq 2d\exp\left(-\frac{cnt^2}{Le^Lt + L^2e^{2L}}\right)\\
     &\leq 2d\exp\left(-\frac{cnt^2}{2L^2e^{2L}}\right),
\end{align*}
where for the second inequality we have assumed that $t\leq Le^L\sqrt{\frac{\log d}{n}}\leq Le^L.$

\end{proof}

\section{Lower Bound}
In this section, we show that the tail bound needs to depend as $L^2e^{O(L)}$ as given in Theorem \ref{thm:mainthm} even for the case of scalars rather than matrices. Consider a two-point distribution which takes values $X_i = 0$ or $X_i = 2L$ with equal probability. $X_i$ can thus be represented as $X_i = L+LY_i$ where $Y_i$ is a Rademacher random variable. Thus $\mathbb{E}[X] = L$. For sufficiently large $n$, $\prod\limits_{i=1}^{n}\left(1+\frac{X_i}{n}\right)= exp\left(\sum\limits_{i=1}^{n}\frac{X_i}{n}\right)(1+o_n(1))$. Taking $t=Le^Lc$, we have:
\begin{align*}
    \mathsf{Pr}\left[exp\left(\sum\limits_{i=1}^{n}\frac{L+LY_i}{n}\right)-e^L\geq cLe^L\right]&=\mathsf{Pr}\left[exp\left(\sum\limits_{i=1}^{n}\frac{LY_i}{n}\right)-1\geq cL\right]\\
    &=\mathsf{Pr}\left[\sum\limits_{i=1}^{n}\frac{LY_i}{n}\geq \log(1+cL)\right]\\
    &\geq \mathsf{Pr}\left[\sum\limits_{i=1}^{n}\frac{LY_i}{n}\geq cL\right]\\
    &\geq \mathsf{Pr}\left[\sum\limits_{i=1}^{n}\frac{Y_i}{n}\geq c\right],
\end{align*}
where the first inequality follows as $\log(1+x) < x$ for sufficiently large $x$ and hence corresponds to a larger probability event. Hence, we obtain a lower bound on the probability which is independent of $L$ and so indeed the $Le^{O(L)}$ term must appear in the tail bound. Here we have $O(L)$ in the exponent because in the lower bound example, the $X_i$ are bounded by $2L$ rather than $L$.

\addreferencesection
\bibliographystyle{amsalpha}
\bibliography{references}

\end{document}